\documentclass[english,10pt]{amsart}

\usepackage{amsrefs}

\usepackage{amssymb,amsmath,amsfonts,amsthm}
\usepackage{graphicx}

\usepackage{graphicx,epic,eepic}
\usepackage{epsfig}
\usepackage{latexsym,esint}
\usepackage{color}

\usepackage[colorlinks=true,citecolor=blue]{hyperref}

\newtheorem{theorem}{Theorem}[section]
\newtheorem{corollary}[theorem]{Corollary}

\newtheorem{lemma}[theorem]{Lemma}

\theoremstyle{remark}
\newtheorem{rmk}{Remark}

\numberwithin{equation}{section}

\def\neweq#1{\begin{equation}\label{#1}}
\def\endeq{\end{equation}}

\def\eps{\varepsilon}

\def\R{{\mathbb R} }

\def\r{\R^{n+1}_{+}}
\def\rn{\R^{n}}
\def\br{\partial\r}
\def\super{\overline}

\begin{document}

\title{On the fractional Lane-Emden equation}

\author[J. D\'avila]{Juan D\'avila}

\address{Departamento de Ingenier\'ia Matem\'atica and CMM, Universidad de Chile, Casilla 170 Correo 3, Santiago, Chile}

\email{jdavila@dim.uchile.cn}

\author[L. Dupaigne]{Louis Dupaigne}

\address{Institut Camille Jordan, UMR CNRS 5208, Universit\'e Claude Bernard Lyon 1, 43 boulevard du 11 novembre 1918, 69622 Villeurbanne cedex, France}

\email{louis.dupaigne@math.cnrs.fr}

\author[J. Wei]{Juncheng Wei}

\address{Department of Mathematics, University of British Columbia, Vancouver, B.C., Canada, V6T 1Z2. }

\email{jcwei@math.ubc.ca}

\thanks{}

\keywords{}

\begin{abstract}We classify solutions of finite Morse index of the fractional Lane-Emden equation 
$$
(-\Delta)^{s} u = \vert u\vert^{p-1}u\quad\text{in $\R^{n}$.}
$$
\end{abstract}

\maketitle

\section{Introduction}
Fix an integer $n\ge 1$ and let $p_{S}(n)$ denote the classical Sobolev exponent:
$$
p_{S}(n)=\left\{
\begin{aligned}
+\infty&\quad\text{if $n\le 2$}\\
\frac{n+2}{n-2}&\quad\text{if $n\ge 3$}
\end{aligned}
\right.
$$
A celebrated result of Gidas and Spruck \cite{gs} asserts that there is no positive solution to the Lane-Emden equation
\begin{equation} \label{le}
-\Delta u = \vert u\vert^{p-1}u\qquad\text{in $\R^n$},
\end{equation}
whenever  $p\in (1,p_{S}(n))$.  For $p=p_{S}(n)$, the same equation is known to have (up to translation and rescaling) a unique positive solution, which is radial and explicit (see Caffarelli-Gidas-Spruck \cite{cgs}).
Let now $p_{c}(n)>p_{S}(n)$ denote the Joseph-Lundgren exponent:
$$
p_{c}(n)=\left\{
\begin{aligned}
+\infty&\quad\text{if $n\le 10$}\\
\frac{(n-2)^2-{4n}+8\sqrt{n-1} }{(n-2)(n-10)}&\quad\text{if $n\ge 11$}
\end{aligned}
\right.
$$
This exponent can be characterized as follows: for $p\ge p_{S}(n)$, the explicit singular solution $u_{s}(x)=A\vert x\vert^{-\frac2{p-1}}$ is unstable if and only if $p<p_c(n)$.
It was proved by  Farina \cite{f} that \eqref{le} has no nontrivial finite Morse index solution whenever $1<p<p_{c}(n)$, $p\neq p_{S}(n)$.

Through blow-up analysis, such Liouville-type theorems imply interior regularity for solutions of a large class of semilinear elliptic equations: they are known to be equivalent to universal estimates for solutions of
\begin{equation} \label{general}
-Lu = f(x,u,\nabla u)\quad\text{in $\Omega$,}
\end{equation}
where $L$ is a uniformly elliptic operator with smooth coefficients, the nonlinearity $f$ scales like $\vert u\vert^{p-1}u$ for large values of $u$, and $\Omega$ is an open set of $\R^n$. For precise statements, see the work of Polacik, Quittner and Souplet \cite{P-Q-S} in the subcritical setting, as well as its adaptation to the supercritical case by Farina and two of the authors \cite{ddf}.

In the present work, we are interested in understanding whether similar results hold for equations involving a nonlocal diffusion operator, the simplest of which is perhaps the fractional laplacian.
Given $s\in(0,1)$, the fractional version of the Lane-Emden equation\footnote{Unlike local diffusion operators, local elliptic regularity for equations of the form \eqref{general} where this time $L$ is the generator of a general Markov diffusion, cannot be captured from the sole understanding of the fractional Lane-Emden equation. For example, further investigations will be needed for operators of L\'evy symbol $\psi(\xi)=\int_{S^{n-1}}\vert\omega\cdot\xi\vert^{2s}\mu(d\omega)$, where $\mu$ is any finite symmetric measure on the sphere $S^{n-1}$.} reads
\begin{equation} \label{p}
(-\Delta)^{s} u = \vert u\vert^{p-1}u\quad\text{in $\R^{n}$.}
\end{equation}
Here we have assumed that $u\in C^{2\sigma}(\R^n)$, $\sigma>s$ and
\begin{equation} \label{tech}
\int_{\R^n}\frac{\vert u(y)\vert}{(1+\vert y\vert)^{n+2s}}\;dy<+\infty,
\end{equation}
so that the fractional laplacian of $u$
$$
(-\Delta)^s u(x) := \mathcal A_{n,s}\int_{\R^n}\frac{u(x)-u(y)}{\vert x-y\vert^{n+2s}}\;dy
$$
is well-defined (in the principal-value sense) at every point $x\in\R^n$.
The normalizing constant $\mathcal A_{n,s} = \frac{2^{2s-1}}{\pi^{n/2}} \frac{\Gamma\left(\frac{n+2s}{2}\right)}{\vert\Gamma(-s)\vert}$ is of the order of $s(1-s)$ as $s$ converges to $0$ or $1$.

The aforementioned classification results of Gidas-Spruck and Caffarelli-Gidas-Spruck have been generalized to the fractional setting (see Y. Li \cite{yy} and Chen-Li-Ou \cite{cho}). The corresponding fractional Sobolev exponent is given by
$$
p_{S}(n)=\left\{
\begin{aligned}
+\infty&\quad\text{if $n\le 2s$}\\
\frac{n+2s}{n-2s}&\quad\text{if $n> 2s$}
\end{aligned}
\right.
$$
Our main result is the following Liouville-type theorem for the fractional Lane-Emden equation.

\begin{theorem}
\label{thmstable} 
Assume that $n\ge1$ and $0<s<\sigma<1$.
Let $u\in C^{2\sigma}(\R^n)\cap L^1(\R^n,(1+\vert y\vert)^{n+2s}dy)$ be a solution to (\ref{p})  which is stable outside a compact set i.e. there exists $R_0\ge0$ such that for every $\varphi\in C^1_{c}(\R^n\setminus \overline{B_{R_{0}}})$,
\begin{equation} \label{stability}
p\int_{\R^n}\vert u\vert^{p-1}\varphi^2\;dx\le \| \varphi \|_{\dot H^s(\R^n)}^2.
\end{equation}

\begin{itemize}
\item If $1<p<p_{S}(n)$ or if $p_{S}(n)<p$ and
\begin{equation} \label{cond}
p \frac{\Gamma(\frac n2-\frac{s}{p-1}) \Gamma(s+\frac{s}{p-1})}{\Gamma(\frac{s}{p-1}) \Gamma(\frac{n-2s}{2} - \frac{s}{p-1})} > \frac{\Gamma(\frac{n+2s}{4})^2}{\Gamma(\frac{n-2s}{4})^2},
\end{equation}
then $ u \equiv 0 $;
\item If $p=p_{S}(n)$, then $u$ has finite energy i.e.
\[\| u\|_{\dot H^s(\R^n)}^2 =\int_{\mathbb R^{n}}|u|^{p+1}<+\infty.\]
If in addition $u$ is stable, then in fact $u\equiv 0$.
\end{itemize}
\end{theorem}

\begin{rmk} For $p>p_{S}(n)$, the function
\begin{align}
\label{sing sol}
u_s(x) = A |x|^{-\frac{2s}{p-1}}
\end{align}
where
$$
A^{p-1} = \lambda\left(\frac{n-2s}{2} -\frac{2s}{p-1}\right)
$$
and where
\begin{equation} \label{lofa}
\lambda(\alpha) = 2^{2s} \frac{\Gamma(\frac{n+2s+2\alpha}{4}) \Gamma(\frac{n+2s-2\alpha}{4})}{\Gamma(\frac{n-2s-2\alpha}{4})\Gamma(\frac{n-2s+2\alpha}{4})}
\end{equation}
is a singular solution to \eqref{p} (see the work by Montenegro and two of the authors \cite{davila-dupaigne-montenegro} for the case $s=1/2$, and the work by Fall \cite[Lemma 3.1]{fall} for the general case). In virtue of the following Hardy inequality (due to Herbst \cite{herbst})
$$
\Lambda_{n,s} \int_{\R^n} \frac{\phi^2}{|x|^{2s}}\;dx\leq \| \phi \|_{\dot H^s(\R^n)}^2
$$
with optimal constant given by
$$
\Lambda_{n,s} = 2^{2s}\frac{\Gamma(\frac{n+2s}{4})^2}{\Gamma(\frac{n-2s}{4})^2},
$$
$u_{s}$ is unstable if only if  \eqref{cond} holds. 
Let us also mention that regular radial solutions in the case $s=1/2$ were constructed by Chipot, Chlebik ad Shafrir \cite{chipot-chlebik-shafrir}. 
Recently, Harada \cite{h}  proved that for $s=1/2$, condition \eqref{cond} is the dividing line for the asymptotic behavior of radial solutions to \eqref{p}, extending thereby the classical results of Joseph and Lundgren \cite{jl} to the fractional Lane-Emden equation in the case $s=1/2$. A similar technique as in \cite{chipot-chlebik-shafrir} allows us to show that the condition \eqref{cond} is optimal. Indeed we have:
\end{rmk}

\begin{theorem}
\label{thm:stable sol}
Assume $p>p_{S}(n)$ and that \eqref{cond} fails.
Then there are radial smooth solutions $u>0$ with $u(r)\to0$ as $r\to\infty$ of \eqref{p} that are stable.
\end{theorem}

It is by now standard knowledge that the fractional laplacian can be seen as a Dirichlet-to-Neumann operator for a degenerate but {\it local} diffusion operator in the higher-dimensional half-space $\R^{n+1}_{+}$:
\begin{theorem}[\cites{spitzer, mo, caffarelli-silvestre-extension}]\label{extension}
Take $s\in(0,1)$, $\sigma>s$ and $u\in C^{2\sigma}(\R^n)\cap L^1(\R^n,(1+\vert y\vert)^{n+2s}dy)$.
For $X=(x,t)\in\R^{n+1}_{+}$, let
$$
\bar u(X)= \int_{\R^n} P(X,y)u(y)\;dy,
$$
where
$$
P(X,y) = p_{n,s}\, t^{2s}\vert X-y\vert^{-(n+2s)}
$$
and $p_{n,s}$ is chosen so that $\int_{\R^n}P(X,y)\;dy=1$.
Then, $\bar u\in C^2(\r)~\cap~C(\super\r)$, $t^{1-2s}\partial_{t}\bar u\in C(\super\r)$ and
\begin{equation*}
\left\{
\begin{aligned}
\nabla\cdot(t^{1-2s}\nabla \bar u)&=0&\quad\text{in $\R^{n+1}_{+}$,}\\
\bar u&= u&\quad\text{on $\br$,}\\
-\lim_{t\to0} t^{1-2s}\partial_{t}\bar u&= \kappa_s(-\Delta)^s u&\quad\text{on $\br$,}
\end{aligned}
\right.
\end{equation*}
where
\begin{equation} \label{kappas}
\kappa_s = \frac{\Gamma(1-s)}{2^{2s-1} \Gamma(s)}.
\end{equation}
\end{theorem}
Applying Theorem \ref{extension} to a solution of the fractional Lane-Emden equation, we end up with the equation
\begin{equation} \label{xp}
\left\{
\begin{aligned}
-\nabla \cdot ( t^{1-2s} \nabla \bar u ) &= 0 &\quad \text{in } \r\\
-\lim_{t\to0} t^{1-2s} \partial_t \bar u &= \kappa_s \vert \bar u\vert^{p-1}\bar u
&\quad \text{on } \br
\end{aligned}
\right.
\end{equation}
An essential ingredient in the proof of Theorem \ref{thmstable} is the following monotonicity formula
\begin{theorem}\label{mf}
Take a solution to \eqref{xp} $\bar u\in C^2(\r)\cap C(\super\r)$ such that $t^{1-2s}\partial_{t}\bar u\in C(\super\r)$. For $x_{0}\in\br$, $\lambda>0$,  let
\begin{multline*}
E(\bar u,x_{0};\lambda) = \\
\lambda^{2s\frac{p+1}{p-1}-n}\left(\frac12\int_{\r\cap B(x_{0},\lambda)} t^{1-2s}\vert\nabla\bar u\vert^2\;dx\,dt - \frac{\kappa_{s}}{p+1}\int_{\br\cap B(x_{0},\lambda)}\vert \bar u\vert^{p+1}\;dx\right)\\
+ \lambda^{2s\frac{p+1}{p-1}-n-1}\frac{s}{p+1}\int_{\partial B(x_{0},\lambda)\cap\r}t^{1-2s}\bar u^2\;d\sigma.
\end{multline*}
Then, $E$ is a nondecreasing function of $\lambda$. Furthermore,
$$
\frac{dE}{d\lambda} = \lambda^{2s\frac{p+1}{p-1}-n+1}\int_{\partial B(x_{0},\lambda)\cap\r}t^{1-2s}\left(\frac{\partial \bar u}{\partial r}+\frac{2s}{p-1}\frac {\bar u}r\right)^2\;d\sigma
$$
\end{theorem}

\begin{rmk} In the above, $B(x_{0},\lambda)$ denotes the euclidean ball in $\r$ centered at $x_{0}$ of radius $\lambda$, $\sigma$  the $n$-dimensional Hausdorff measure restricted to $\partial B(x_{0},\lambda)$, $r=\vert X\vert$  the euclidean norm of a point $X=(x,t)\in~\r$, and $\partial_{r} = \nabla\cdot\frac{X}r$  the corresponding radial derivative. 
\end{rmk}

An analogous monotonicity formula has been derived  by Fall and Felli \cite{fall-felli} to obtain unique continuation results for fractional equations. Previously, Caffarelli and Silvestre derived an Almgren quotient formula for the fractional laplacian in \cite{caffarelli-silvestre-extension} and Caffarelli, Roquejoffre and Savin \cite{carfarelli-roquejoffre-savin} obtained a related monotonicity formula to study regularity of nonlocal minimal surfaces. Another monotonicity formula for fractional problems was obtained by Cabr\'e and Sire \cite{cabre-sire} and used by Frank, Lenzmann and Silvestre \cite{frank-lenzmann-silvestre}.

The proof of Theorem \ref{thmstable} follows an approach used in our earlier work with Kelei Wang \cite{ddww} (see also \cite{kw}). First we derive suitable energy estimate (Section~\ref{sect:Energy estimates}) and handle the critical and subcrtiicical cases (Section~\ref{sect:The subcritical case}). In Section~\ref{sect:The monotonicity formula} we give a proof of the monotonicity formula  Theorem~\ref{mf}.
Then we use the monotonicity formula  and a blown-down analysis (Section~\ref{sect:Blow-down analysis}) to reduce to homogeneous singular solutions. We exclude the existence of stable homogeneous singular solutions in the optimal range of $p$ (Section~\ref{sect:Homogeneous solutions}). Finally we prove Theorem~\ref{thm:stable sol} in Section~\ref{sect:Construction of radial entire stable solutions}.

\section{Energy estimates}
\label{sect:Energy estimates}
\begin{lemma} \label{cap} Let $u$ be a solution to \eqref{p}. Assume that $u$ is stable outside some ball $B_{R_{0}}\subset\R^n$. Let $\eta\in C^\infty_{c}(\R^n\setminus \super{B_{R_{0}}})$ and for $x\in\R^n$, define
\begin{equation} \label{def rho}
\rho(x) = \int_{\rn}\frac{(\eta(x)-\eta(y))^2}{\vert x-y\vert^{n+2s}}\;dy
\end{equation}
Then,
$$
\int_{\rn}\vert u\vert^{p+1}\eta^{2}\,dx + \frac{1}p\| u\eta \|_{\dot H^{s}(\rn)}^2
\le \frac2{p-1}\int_{\rn}u^2\rho\,dx.
$$
\end{lemma}

\proof
Multiply
\eqref{p} by $u\eta^{2}$.
Then,
\begin{align*}
\int_{\rn}\vert u\vert^{p+1}\eta^{2}\;dx &= \int_{\rn} (-\Delta)^{s}u\, u\eta^{2}\;dx\\
&=\int_{\rn}\int_{\rn} \frac{(u(x)-u(y))(u(x)\eta(x)^{2}-u(y)\eta(y)^{2})}{\vert x-y\vert^{n+2s}}\;dx\,dy\\
&=\int_{\rn}\int_{\rn}
\frac{u^2(x)\eta^{2}(x) -u(x)u(y)(\eta^{2}(x)+\eta^{2}(y))+u^2(y)\eta^{2}(y)}{\vert x-y\vert^{n+2s}}\;dx\,dy\\
&=\int_{\rn}\int_{\rn}
\frac{(u(x)\eta(x)-u(y)\eta(y))^{2}-(\eta(x)-\eta(y))^2u(x)u(y)}{\vert x-y\vert^{n+2s}}\;dx\,dy\\
&= \| u\eta\|_{\dot H^{s}(\rn)}^{2} - \int_{\rn}\int_{\rn}\frac{(\eta(x)-\eta(y))^2u(x)u(y)}{\vert x-y\vert^{n+2s}}\;dx\,dy
\end{align*}
Using the inequality $2ab\le a^{2}+ b^{2}$, we deduce that
\begin{equation} \label{t1}
\| u\eta\|_{\dot H^{s}(\rn)}^{2} - \int_{\rn}\vert u\vert^{p+1}\eta^{2}\;dx\le\int_{\rn}u^2\rho\;dx
\end{equation}
Since $u$ is stable, we deduce that
$$
(p-1)\int_{\rn}\vert u\vert^{p+1}\eta^{2}\;dx\le\int_{\rn}u^2\rho\;dx
$$
Going back to \eqref{t1}, it follows that
$$
\frac{1}{p}\| u\eta\|_{\dot H^{s}(\rn)}^{2} + \int_{\rn}\vert u\vert^{p+1}\eta^{2}\;dx\le\frac2{p-1}\int_{\rn}u^2\rho\;dx
$$
\hfill\qed

\begin{lemma} \label{lt} For  $m>n/2$ and $x\in\rn$, let
\begin{equation} \label{def eta}
\eta(x) = (1+\vert x\vert^{2})^{-m/2}\qquad\text{and}\quad\rho(x) = \int_{\rn}\frac{(\eta(x)-\eta(y))^2}{\vert x-y\vert^{n+2s}}\;dy
\end{equation}
Then, there exists a constant $C=C(n,s,m)>0$ such that
\begin{equation} \label{t2}
C^{-1}\left(1+\vert x\vert^2\right)^{-\frac n2-s}\le\rho(x)\le C\left(1+\vert x\vert^2\right)^{-\frac n2-s}.
\end{equation}
\end{lemma}

\proof
Let us prove the upper bound first.
Since $\rho$ is a continuous function, we may always assume that $\vert x\vert\ge1$. Split the integral
$$
\int_{\R^{n}} \frac{(\eta(x)-\eta(y))^2}{\vert x-y\vert^{n+2s}}\;dy
$$
in the regions $|x-y|< |x|/2$, $|x|/2 \leq |x-y|\leq 2|x|$,  and $|x-y|> 2|x|$.
When $|x-y|\leq |x|/2$,
$$
|\eta(x)-\eta(y)|\leq C (1+|x|^2)^{-\frac{m+1}2} |x-y|.
$$
So,
\begin{align*}
\int_{|x-y|\leq |x|/2 }
\frac{(\eta(x)-\eta(y))^2}{\vert x-y\vert^{n+2s}}\;dy
\leq
C(1+|x|^2)^{-m-1}  \int_{|x-y|\leq |x|/2 } |x-y|^{2 -n-2s}\;dy\\
\leq C(1+|x|^2)^{-m-s}\le C\left(1+\vert x\vert^2\right)^{-\frac n2-s}.
\end{align*}
When $|x|/2 \leq |x-y|\leq 2|x|$,
\begin{align*}
\int_{|x|/2 \leq |x-y|\leq 2|x|}
\frac{(\eta(x)-\eta(y))^2}{\vert x-y\vert^{n+2s}}\;dy
\leq C|x|^{-n-2s} \int_{|y|\leq 2|x|} ( \eta(x)^2 + \eta(y)^2 ) \; dy\\
\leq C|x|^{-n-2s} ( |x|^{-2m + n} + 1 )  \leq C(1+|x|^2)^{-\frac n2-s},
\end{align*}
where we used the assumption $m>\frac n2$.
When $|x-y|> 2|x|$, then $|y|\geq |x|$ and $\eta(y)\leq C (1+|x|^2)^{-m/2}$. Then,
\begin{multline*}
\int_{|x-y| > 2|x| }
\frac{(\eta(x)-\eta(y))^2}{\vert x-y\vert^{n+2s}}\;dy
\leq
C(1+|x|^2)^{-m} \int_{|x-y| > 2|x| }
\frac{1}{\vert x-y\vert^{n+2s}}\;dy
\\\leq C (1+|x|^2)^{-m-s} \leq C(1+|x|^2)^{-\frac n2-s}.
\end{multline*}
Let us turn to the lower bound. Again, we may always assume that $\vert x\vert\ge1$. Then,
$$
\rho(x)\ge \int_{|y|\leq 1/2 }
\frac{(\eta(y)-\eta(x))^2}{\vert x-y\vert^{n+2s}}\;dy\ge \left(\frac{\vert x\vert}{2}\right)^{-(n+2s)}\int_{|y|\leq 1/2 }
{(\eta(y)-2^{-m/2})^2}\;dy
$$
and the estimate follows.
\hfill\qed

\begin{corollary}\label{coro en}  Let $m>n/2$, $\eta$ given by \eqref{def eta}, $R\ge R_{0}\ge1$, $\psi\in C^\infty(\R^n)$ be such that $0\le\psi\le1$, $\psi\equiv0$ on $B_{1}$ and $\psi\equiv1$ on $\R^n\setminus B_{2}$. Let
\begin{equation} \label{def eta R}
\eta_{R}(x) = \eta\left(\frac xR\right)\psi\left(\frac x{R_0}\right)\qquad\text{and}\quad\rho_{R}(x) = \int_{\rn}\frac{(\eta_{R}(x)-\eta_{R}(y))^2}{\vert x-y\vert^{n+2s}}\;dy
\end{equation}
There exists a constant $C=C(n,s,m,R_{0})>0$ such that for all 
$\vert x\vert\ge 3R_{0}$
$$
\rho_{R}(x)\le
C\eta\left(\frac xR\right)^2 \vert x\vert^{-(n+2s)}+
R^{-2s}\rho\left(\frac xR\right)
$$
\end{corollary}

\proof 
Fix $x$ such that $\vert x\vert\ge R\ge 3R_{0}$. Using the definition of $\eta_{R}$ and Young's inequality, we have
\begin{align*}
\frac12 \rho_{R}(x)&\le   \eta\left(\frac xR\right)^2
\int_{\rn}\frac{\left(\psi\left(\frac x{R_0}\right)-\psi\left(\frac y{R_0}\right)\right)^2}{\vert x-y\vert^{n+2s}}\;dy +
\int_{\rn}\psi\left(\frac y{R_0}\right)^2\frac{(\eta \left(\frac xR\right)-\eta\left(\frac yR\right))^2}{\vert x-y\vert^{n+2s}}\;dy\\
&\le  \eta\left(\frac xR\right)^2
\int_{B_{2R_0}}\frac{1}{\vert x-y\vert^{n+2s}}\;dy +
\int_{\rn}\frac{(\eta \left(\frac xR\right)-\eta\left(\frac yR\right))^2}{\vert x-y\vert^{n+2s}}\;dy\\
&\le  C\eta\left(\frac xR\right)^2 \vert x\vert^{-(n+2s)}+
R^{-2s}\rho\left(\frac xR\right)
\end{align*}
and the result follows.
\qed



\begin{lemma}\label{lemma energy} Let $u$ be a solution of \eqref{p} which is stable outside a ball $B_{R_{0}}$. Take $\rho_{R}$ as in Corollary \ref{coro en} with $m\in(\frac n2,\frac n2+\frac{s(p+1)}2)$. Then, there exists a constant $C=C(n,s,m,p,R_{0})>0$ such that   for all $R\ge 3R_{0}$,
$$
\int_{\rn}u^{2}\rho_{R}\,dx
\le C\left(\int_{B_{3R_{0}}}u^{2}\rho_{R}\,dx +R^{n-2s\frac{p+1}{p-1}}\right).
$$
\end{lemma}

\proof By Corollary \ref{coro en}, if $R\ge \vert x\vert \ge 3R_{0}$, then
$$
\rho_{R}(x)\le C(\vert x\vert^{-n-2s}+R^{-2s})
$$
and so
$$
\int_{B_{R}\setminus B_{3R_{0}}} \rho_{R}(x)^{\frac{p+1}{p-1}}\eta_{R}(x)^{-\frac{4}{p-1}}\;dx\le CR^{n-2s\frac{p+1}{p-1}}.
$$
Similarly, if $\vert x\vert\ge R\ge 3{R_{0}}$, then
$$
\rho_{R}(x)\le C R^{-2s}\left(1+\frac{\vert x\vert^2}{R^2}\right)^{-\frac n2-s}
$$
and so
$$
\rho_{R}(x)^{\frac{p+1}{p-1}}\eta_{R}(x)^{-\frac{4}{p-1}} \le C R^{-2s\frac{p+1}{p-1}}
\left(1+\frac{\vert x\vert^2}{R^2}\right)^{-\frac{n+2s}2\frac{p+1}{p-1}+\frac{2m}{p-1}}
$$
Since $m\in(\frac n2,\frac n2+s\frac{p+1}2)$, we have $\frac{2m}{p-1}-\frac{n+2s}2\frac{p+1}{p-1}<-\frac n2$ and so
$$
\int_{\R^n\setminus B_{3R_{0}}} \rho_{R}(x)^{\frac{p+1}{p-1}}\eta_{R}(x)^{-\frac{4}{p-1}}\;dx\le CR^{n-2s\frac{p+1}{p-1}}.
$$
Now,
\begin{multline*}
\int_{\rn}u^2\rho_{R}\;dx= \int_{B_{3R_{0}}}u^2\rho_{R}\;dx+
 \int_{\rn\setminus B_{3R_{0}}}u^2\rho_{R}\, \eta_{R}^{-\frac4{p+1}}\, \eta_{R}^{\frac4{p+1}}\,dx\\
\le  \int_{B_{3R_{0}}}u^2\rho_{R}\;dx+ \left(\int_{\rn}\vert u\vert^{p+1}\eta_{R}^{2}\;dx\right)^{\frac2{p+1}}
\left(
\int_{\rn}
\rho_{R}^{\frac{p+1}{p-1}}\eta_{R}^{-\frac4{p-1}}\;dx
\right)^{\frac{p-1}{p+1}}\\
\le  \int_{B_{3R_{0}}}u^2\rho_{R}\;dx+ C R^{(n-2s\frac{p+1}{p-1})\frac{p-1}{p+1}}\left(\int_{\rn}\vert u\vert^{p+1}\eta_{R}^{2}\;dx\right)^{\frac2{p+1}}
\end{multline*}
By a standard approximation argument, Lemma \ref{cap} remains valid with $\eta=\eta_{R}$ and $\rho=\rho_{R}$ and so the result follows.
\hfill\qed

\begin{lemma} \label{l2esti} Assume that $p\ne\frac{n+2s}{n-2s}$. Let $u$ be a solution to \eqref{p} which is stable outside a ball $B_{R_{0}}$ and $\bar u$ its extension, solving \eqref{xp}. Then, there exists a constant $C=C(n,p,s,R_{0},u)>0$ such that
$$
\int_{B_{R}}t^{1-2s}\bar u^2\;dx dt \le C R^{n+2(1-s)-\frac{4s}{p-1} }
$$
for any $R>3R_0$.
\end{lemma}

\proof

According to Theorem \ref{extension},
$$
\bar u(x,t) = p_{n,s} \int_{\R^n} u(z) \frac{t^{2s}}{(|x-z|^2+t^2)^{\frac{n+2s}2}} \; dz
$$
so that
$$
\bar u(x,t)^2
\leq p_{n,s}
\int_{\R^n} u(z)^2 \frac{t^{2s}}{(|x-z|^2+t^2)^{\frac{n+2s}2}} \; dz .
$$
So,
\begin{multline*}
\int_{B_{R}}t^{1-2s}\bar u^2\;dx dt \le p_{n,s}\int_{\vert x\vert\le R,z\in\R^n}u(z)^2\left(\int_{0}^R  \frac{t}{(|x-z|^2+t^2)^{\frac{n+2s}2}}\; dt\right)dzdx
\\
\le C\int_{\vert x\vert\le R, z\in\R^n}u^2(z) \left\{\left(\vert x-z\vert^2\right)^{-\frac n2+1-s} - \left(\vert x-z\vert^2+R^2\right)^{-\frac n2+1-s}\right\}\;dzdx
\end{multline*}
Split this last integral according to $\vert x-z\vert<2R$ or  $\vert x-z\vert\ge 2R$. Then,
\begin{multline*}
\int_{\vert x\vert\le R, \vert x-z\vert<2R}u^2(z) \left\{\left(\vert x-z\vert^2\right)^{-\frac n2+1-s} - \left(\vert x-z\vert^2+R^2\right)^{-\frac n2+1-s}\right\}\;dzdx\le \\
\int_{\vert x\vert\le R,\vert x-z\vert<2R}u^2(z) \left(\vert x-z\vert^2\right)^{-\frac n2+1-s}\;dzdx\le C R^{2(1-s)}\int_{B_{3R}}u^2(z)\;dz\le\\
CR^{2(1-s)}\left(\int\vert u\vert^{p+1}\eta_{R}^2\right)^{\frac2{p+1}}\left(\int_{B_{3R}}\eta_{R}^{-\frac{4}{p-1}}\right)^{\frac{p-1}{p+1}}\le
\\
C R^{2(1-s)+n\frac{p-1}{p+1}}\left(\int u^2(z)\rho_{R}(z)\;dz\right)^{\frac2{p+1}}\le C R^{n+2(1-s)-\frac{4s}{p-1} },
\end{multline*}
where we used H\"older'sin equality, then Lemma \ref{cap} and then  Lemma \ref{lemma energy}.
For the region $\vert x-z\vert\ge 2R$, the mean-value inequality implies that
\begin{multline*}
\int_{\vert x\vert\le R, \vert x-z\vert\ge 2R}u^2(z) \left\{\left(\vert x-z\vert^2\right)^{-\frac n2+1-s} - \left(\vert x-z\vert^2+R^2\right)^{-\frac n2+1-s}\right\}\;dzdx\le \\
CR^2\int_{\vert x\vert\le R, \vert x-z\vert\ge 2R}u^2(z) \vert x-z\vert^{-(n+2s)}\;dzdx\le CR^{n+2}\int_{\vert z\vert\ge R}u^2(z) \vert z\vert^{-(n+2s)}\;dz\\
\le CR^{2}\int_{\vert z\vert\ge R}u^2\rho\;dz\le C R^{n+2(1-s)-\frac{4s}{p-1} },
\end{multline*}
where we used again Corollary \ref{coro en} in the penultimate inequality and Lemma \ref{lemma energy}  in the last one.
\hfill\qed

\begin{lemma} \label{lj} Let $u$ be a solution to \eqref{p} which is stable outside a ball $B_{R_{0}}$ and $\bar u$ its extension, solving \eqref{xp}. Then, there exists a constant $C=C(n,p,s,u)>0$ such that
$$
\int_{B_{R}\cap\r}t^{1-2s}\vert\nabla \bar u\vert^{2}\;dx\,dt + \int_{B_{R}\cap\br}\vert u\vert^{p+1}\,dx\le C R^{n-2s\frac{p+1}{p-1}}
$$
\end{lemma}

\proof
The $L^{p+1}$ estimate follows from Lemmata \ref{cap} and \ref{lemma energy}.  Now take a cut-off function $\eta\in C^1_{c}(\super\r)$ such that $\eta=1$ on $\r\cap(B_R \setminus B_{2R_{0}})$
and $\eta = 0$ on $B_{R_{0}}\cup (\r\setminus B_{2R})$,
and multiply equation \eqref{xp} by $\bar u \eta^{2}$. Then,
\begin{align}
\kappa_{s}\int_{\br} \vert \bar u\vert^{p+1}\eta^{2}\;dx &= \int_{\r}t^{1-2s}\left\{
\nabla\bar u\cdot\nabla (\bar u\eta^{2})
\right\}\;dx\,dt\nonumber\\
&=\int_{\r}t^{1-2s}\left\{
\vert\nabla (\bar u\eta)\vert^{2} - \bar u^{2}\vert\nabla\eta\vert^{2}
\right\}\;dx\,dt .\label{last}
\end{align}
Since $u$ is stable outside $B_{R_{0}}$, so is $\bar u$ and we deduce that
$$
\frac1p \int_{\r}t^{1-2s}
\vert\nabla (\bar u\eta)\vert^{2} \;dx\,dt\ge \int_{\r}t^{1-2s}\left\{
\vert\nabla (\bar u\eta)\vert^{2} - \bar u^{2}\vert\nabla\eta\vert^{2}
\right\}\;dx\,dt.
$$
In other words,
\begin{align}
\label{grad}
p'\int_{\r}t^{1-2s} \bar u^{2}\vert\nabla\eta\vert^{2}
\;dx\,dt\ge
\int_{\r}t^{1-2s}
\vert\nabla (\bar u\eta)\vert^{2} \;dx\,dt,
\end{align}
where $\frac1{p'}+\frac1p=1$. We then apply Lemma \ref{l2esti}.
\qed

\section{The subcritical case}
\label{sect:The subcritical case}
In this section, we prove Theorem \ref{thmstable} for $1<p\le p_{S}(n)$.

\proof Take a solution $u$ which is stable outside some ball $B_{R_{0}}$. Apply Lemma \ref{lemma energy} and let $R\to+\infty$. Since $p\le p_{S}(n)$, we deduce that $u\in \dot H^s(\R^n)\cap L^{p+1}(\R^n)$. Multiplying the equation \eqref{p} by $u$ and integrating, we deduce that
\begin{equation} \label{eggy}
\int_{\R^n}\vert u\vert^{p+1} = \| u \|_{\dot H^s(\R^n)}^2,
\end{equation}
while multiplying by $u^\lambda$ given for $\lambda>0$ and $x\in\R^n$ by
$$
u^{\lambda}(x) = u(\lambda x)
$$
yields
$$
\int_{\R^n} \vert u \vert^{p-1}u^\lambda = \int_{\R^n} (-\Delta)^{s/2}u(-\Delta)^{s/2}u^\lambda = \lambda^s\int_{\R^n}w w^\lambda,
$$
where $w=(-\Delta)^{s/2}u$. Following Ros-Oton and Serra \cite{rs2}, we use the change of variable $y=\sqrt\lambda\, x$ to deduce that
$$
\lambda^s\int_{\R^n}w w^\lambda \;dx = \lambda^{\frac{2s-n}{2}} \int_{\R^n}w^{\sqrt\lambda}w^{1/\sqrt\lambda}\;dy
$$
Hence,
\begin{multline*}
-\frac n{p+1}\int_{\R^n}\vert u\vert^{p+1} = \int_{\R^n}x\cdot\nabla \frac{\vert u\vert^{p+1}}{p+1}=\int_{\R^n}(\vert u\vert^{p-1}u) x\cdot\nabla u=\\
\left.\frac d{d\lambda}\right\vert_{\lambda=1}\int_{\R^n} \vert u \vert^{p-1}uu^\lambda = \left.\frac d{d\lambda}\right\vert_{\lambda=1}\lambda^{\frac{2s-n}{2}} \int_{\R^n}w^{\sqrt\lambda}w^{1/\sqrt\lambda}\;dy=\\
\frac{2s-n}{2}\int_{\R^n}w^2 + \left.\frac d{d\lambda}\right\vert_{\lambda=1} \int_{\R^n}w^{\sqrt\lambda}w^{1/\sqrt\lambda}\;dy= \frac{2s-n}{2}\| u \|_{\dot H^s(\R^n)}^2
\end{multline*}
In the last equality, we have used the fact that $w\in C^1(\R^n)$, as follows by elliptic regularity. We have just proved the following Pohozaev identity
$$
\frac n{p+1}\int_{\R^n}\vert u\vert^{p+1} = \frac{n-2s}{2}\| u \|_{\dot H^s(\R^n)}^2
$$
For $p<p_{S}(n)$, the above  identity together with \eqref{eggy} force $u\equiv0$. For $p=p_{S}(n)$, we are left with proving that there is no stable nontrivial solution. Since $u\in \dot H^s(\R^n)$, we may apply the stability inequlatiy \eqref{stability} with test function $\varphi=u$, so that
$$
p\int_{\R^n}\vert u\vert^{p+1} \le \| u \|_{\dot H^s(\R^n)}^2 .
$$
This contradicts \eqref{eggy} unless $u\equiv0$.
\hfill\qed

In the following sections, we present several tools to study the supercritical case.
\section{The monotonicity formula}
\label{sect:The monotonicity formula}
In this section, we prove Theorem \ref{mf}.
\proof
Since the equation is invariant under translation, it suffices to consider the case where the center of the considered ball is the origin $x_{0}=0$. Let
\begin{equation} \label{E1}
\begin{aligned}
&E_{1}(\bar u;\lambda) = \\
&\lambda^{2s\frac{p+1}{p-1}-n}\left(\int_{\r\cap B_\lambda} t^{1-2s}\frac{\vert\nabla\bar u\vert^2}{2}dx\,dt - \int_{\br\cap B_\lambda}\frac{\kappa_{s}}{p+1}\vert \bar u\vert^{p+1}dx\right)
\end{aligned}
\end{equation}
For $X\in\r$, let also
\begin{equation} \label{U} U(X;\lambda) = \lambda^{\frac{2s}{p-1}}\bar u(\lambda X).
\end{equation}
Then, $U$ satisfies the three following properties: $U$ solves \eqref{xp},
\begin{equation}\label{t3}
E_{1}(\bar u;\lambda) = E_{1}(U; 1),
\end{equation}
and, using subscripts to denote partial derivatives,
\begin{equation} \label{t4}
\lambda U_{\lambda} = \frac{2s}{p-1} U + rU_{r}.
\end{equation}
Differentiating the right-hand side of \eqref{t3}, we find
$$
\frac{dE_{1}}{d\lambda}(\bar u;\lambda) = \int_{\r\cap B_1}t^{1-2s}\nabla U\cdot\nabla U_{\lambda}\;dx\,dt -\kappa_{s}\int_{\br\cap B_1}\vert U\vert^{p-1}U_{\lambda}\;dx.
$$
Integrating by parts and then using \eqref{t4},
\begin{align*}
\frac{dE_{1}}{d\lambda}(\bar u;\lambda) &= \int_{\partial B_1\cap\r}t^{1-2s} U_{r}U_{\lambda}\;d\sigma\\
&= \lambda \int_{\partial B_1\cap\r}t^{1-2s} U_{\lambda}^2\;d\sigma - \frac{2s}{p-1}\int_{\partial B_1\cap\r}t^{1-2s} UU_{\lambda}\;d\sigma\\
&= \lambda \int_{\partial B_1\cap\r}t^{1-2s} U_{\lambda}^2\;d\sigma - \frac{s}{p-1}\left(\int_{\partial B_1\cap\r}t^{1-2s} U^2\;d\sigma\right)_{\lambda}\\
\end{align*}
Scaling back, the theorem follows.
\hfill\qed

\section{Homogeneous solutions}
\label{sect:Homogeneous solutions}
\begin{theorem}\label{h} Let $\bar u$ be a stable homogeneous solution of \eqref{xp}. Assume that $p>\frac{n+2s}{n-2s}$ and
\begin{equation} \label{jl}
p \frac{\Gamma(\frac n2-\frac{s}{p-1}) \Gamma(s+\frac{s}{p-1})}{\Gamma(\frac{s}{p-1}) \Gamma(\frac{n-2s}{2} - \frac{s}{p-1})} > \frac{\Gamma(\frac{n+2s}{4})^2}{\Gamma(\frac{n-2s}{4})^2}.
\end{equation}
Then, $\bar u\equiv0$.
\end{theorem}

\proof Take standard polar coordinates in $\r$: $X=(x,t)=r\theta$, where $r=\vert X\vert$ and $\theta=\frac X{\vert X\vert}$. Let  $\theta_{1}= \frac t{\vert X\vert}$ denote the component of $\theta$ in the $t$ direction and $S^n_{+}=\{X\in\r : r =1, \, \theta_{1}>0\}$ denote the upper unit half-sphere.

\noindent{\bf Step 1.} Let $\bar u$ be a homogeneous solution of \eqref{xp} i.e. assume that for some $\psi\in C^2(S^n_{+})$,
$$
\bar u (X) = r^{-\frac{2s}{p-1}}\psi(\theta).
$$
Then,
\begin{align}
\label{mult}
\int_{S^n_+} \theta_1^{1-2s} |\nabla \psi|^2
+\beta
\int_{S^n_+} \theta_1^{1-2s} \psi^2
= \kappa_{s} \int_{\partial S^n_+} \vert\psi\vert^{p+1} ,
\end{align}
where $\kappa_{s}$ is given by \eqref{kappas} and
$$
\beta = \frac{2s}{p-1}\left(n-2s-\frac{2s}{p-1}\right).
$$
Indeed, since $\bar u$ solves \eqref{xp} and is homogeneous, $\psi$ solves
\begin{align}
\label{eq psi}
\left\{
\begin{aligned}
&
-{\rm div}(\theta_1^{1-2s} \nabla \psi) + \beta \theta_1^{1-2s} \psi = 0
\quad
\text{on } S^n_+
\\
&
-\lim_{\theta_{1}\to0}\theta_1^{1-2s} \partial_{\theta_1} \psi = \kappa_{s} \vert\psi\vert^{p-1}\psi
\quad \text{on }  \partial S^n_+,
\end{aligned}
\right.
\end{align}
Multiplying \eqref{eq psi} by $\psi$ and integrating, \eqref{mult} follows.

\noindent{\bf Step 2.} For all $\varphi\in C^1(S^n_{+})$,
 \begin{align}
\label{stab1}
\kappa_{s} p \int_{\partial S^n_+} \vert\psi\vert^{p-1} \varphi^2
\leq
\int_{S^n_+} \theta_1^{1-2s} |\nabla \varphi|^2
+\left(\frac{n-2s}{2}\right)^2
\int_{S^n_+} \theta_1^{1-2s} \varphi^2
\end{align}
By definition, $\bar u$ is stable if for all $\phi\in C^1_{c}(\super\r)$,
\begin{equation} \label{stabibi}
\kappa_{s}p\int_{\br}\vert \bar u\vert^{p-1}\phi^2\;dx
\le\int_{\r}t^{1-2s}\vert\nabla\phi\vert^2\;dxdt
\end{equation}
Choose a standard cut-off function $\eta_{\epsilon}\in C^1_{c}(\mathbb R^*_+)$ at the origin and at infinity i.e. $\chi_{(\epsilon,1/\epsilon)}(r)\le\eta_{\epsilon}(r)\le\chi_{(\epsilon/2,2/\epsilon)}(r)$. Let also $\varphi\in C^1(S^n_{+})$, apply \eqref{stabibi} with
$$
\phi(X) = r^{-\frac{n-2s}{2}}\eta_{\epsilon}(r)\varphi(\theta)\qquad\text{for $X\in\r$,}
$$
and let $\epsilon\to0$. Inequality \eqref{stab1} follows.

\noindent{\bf Step 3.}
For  $\alpha \in (0,\frac{n-2s}{2})$, $x\in\R^n\setminus\{0\}$, let
$$
v_{\alpha}(x) = |x|^{-\frac{n-2s}{2}+\alpha}
$$
and $\bar v_{\alpha}$ its extension, as defined in Theorem \ref{extension}. Then, $\bar v_{\alpha}$ is homogeneous i.e. there exists $\phi_{\alpha}\in C^2(S^n_{+})$ such that for $X\in\r\setminus\{0\}$,
$$
\bar v_{\alpha}(X) = r^{-\frac{n-2s}{2}+\alpha}\phi_{\alpha}(\theta).
$$
In addition, for all $\varphi\in C^1(S^n_{+})$,
\begin{multline}
\label{halpha}
\int_{S^n_+} \theta_1^{1-2s} |\nabla \varphi|^2
+\left(\left(\frac{n-2s}2\right)^2 -\alpha^2\right)
\int_{S^n_+} \theta_1^{1-2s} \varphi^2
\\= \kappa_{s} \lambda(\alpha) \int_{\partial S^n_+}
\varphi^2
+\int_{S^n_+} \theta_1^{1-2s} \phi_\alpha^2 \left|\nabla \left(\frac\varphi{\phi_\alpha}\right)\right|^2
\end{multline}
Indeed, according to Fall \cite[Lemma 3.1]{fall}, $\bar v_{\alpha}$ is homogeneous.
Using the calculus identity stated by Fall-Felli in \cite[Lemma 2.1]{fall-felli}, we get
\begin{align}
\label{eq phi alpha}
\left\{
\begin{aligned}
&
-{\rm div}(\theta_1^{1-2s} \nabla \phi_\alpha) + \left(\left(\frac{n-2s}2\right)^2 -\alpha^2\right) \theta_1^{1-2s} \phi_\alpha = 0
\quad
\text{on } S^n_+
\\
& \phi_\alpha = 1
\quad \text{on }  \partial S^n_+.
\end{aligned}
\right.
\end{align}
Multiply equation \eqref{eq phi alpha} by $\varphi^2/\phi_{\alpha}$, integrate by parts, apply the calculus identity
$$
\nabla\phi_{\alpha}\cdot\nabla\frac{\varphi^2}{\phi_{\alpha}} = \vert\nabla \varphi\vert^2 - \left\vert\nabla\frac{\varphi}{\phi_{\alpha}}\right\vert^2\phi_{\alpha}^2
$$
and recall from Fall \cite[Lemma 3.1]{fall} that
$$
-\lim_{t\to0} t^{1-2s} \partial_t \super v_{\alpha} = \kappa_s \lambda(\alpha) |x|^{-\frac{n-2s}{2}+\alpha -2s},
$$
where
$
\lambda(\alpha)
$
 is given by \eqref{lofa}.

\noindent{\bf Step 4.} For $\alpha \in (0,\frac{n-2s}2)$
\begin{align}
\label{ineq phi alpha}
\phi_0 \leq \phi_\alpha
\quad\text{on } S^n_+.
\end{align}
Indeed,  on $ S^n_+$,
$$
{\rm div}(\theta_1^{1-2s} \nabla \phi_0) = \left(\frac{n-2s}2\right)^2\theta_1^{1-2s} \phi_0
\geq  \left(\left(\frac{n-2s}2\right)^2 -\alpha^2\right) \theta_1^{1-2s} \phi_0
$$
so $\phi_0$ is a sub-solution of \eqref{eq phi alpha}. By the maximum principle, the conclusion follows.

\noindent{\bf Step 5.} End of proof. Fix $\alpha \in (0,\frac{n-2s}2)$ given by
$$
\alpha = \frac{n-2s}{2} - \frac{2s}{p-1}
$$
so that
$$
\left(\frac{n-2s}{2}\right)^2 - \alpha^2 = \frac{2s}{p-1}\left(n-2s - \frac{2s}{p-1}\right)=\beta,
$$
where $\beta$ is the constant appearing in \eqref{eq psi}.

Use the stability inequality \eqref{stab1} with $\varphi = \frac{\psi \phi_0}{\phi_\alpha}$:
\begin{equation} \label{staboo}
\kappa_{s} p \int_{\partial S^n_+} \vert\psi\vert^{p+1}
\leq
\int_{S^n_+} \theta_1^{1-2s} \left|\nabla \left(\frac{\psi \phi_0}{\phi_\alpha}\right)\right|^2
+\left(\frac{n-2s}{2}\right)^2
\int_{S^n_+} \theta_1^{1-2s} \left(\frac{\psi \phi_0}{\phi_\alpha}\right)^2.
\end{equation}
Note that a particular case of the identity \eqref{halpha} is
\begin{align}
\label{h1}
\int_{S^n_+} \theta_1^{1-2s} |\nabla \varphi|^2
+\left(\frac{n-2s}{2}\right)^2
\int_{S^n_+} \theta_1^{1-2s} \varphi^2
= \kappa_{s} \Lambda_{n,s}\int_{\partial S^n_+}
\varphi^2
+\int_{S^n_+} \theta_1^{1-2s} \phi_0^2 \left|\nabla \left(\frac\varphi{\phi_0}\right)\right|^2
\end{align}
Using \eqref{h1} (with $\varphi = \frac{\psi \phi_0}{\phi_\alpha}$), \eqref{staboo} becomes
$$
\kappa_{s} p \int_{\partial S^n_+}\vert\psi\vert^{p+1}
\leq
\kappa_{s} \Lambda_{n,s}\int_{\partial S^n_+}
\psi^2
+\int_{S^n_+} \theta_1^{1-2s} \phi_0^2 \left|\nabla \left(\frac\psi{\phi_\alpha}\right)\right|^2 .
$$
By \eqref{ineq phi alpha}, we deduce that
$$
\kappa_{s} p \int_{\partial S^n_+} \vert\psi\vert^{p+1}
\leq
\kappa_{s} \Lambda_{n,s}\int_{\partial S^n_+}
\psi^2
+\int_{S^n_+} \theta_1^{1-2s} \phi_\alpha^2 \left|\nabla \left(\frac\psi{\phi_\alpha}\right)\right|^2 .
$$
Using again the identity \eqref{halpha}, we deduce that
$$
\kappa_{s} p \int_{\partial S^n_+} \vert\psi\vert^{p+1}
\leq
\kappa_{s} (\Lambda_{n,s}-\lambda(\alpha))\int_{\partial S^n_+}
\psi^2
+
\int_{S^n_+} \theta_1^{1-2s} |\nabla \psi|^2
+\beta
\int_{S^n_+} \theta_1^{1-2s} \psi^2
$$
Comparing with \eqref{mult}, it follows that
\begin{align}
\label{ineq1}
(p-1) \int_{\partial S^n_+}  \vert\psi\vert^{p+1}
\leq
 ( \Lambda_{n,s} - \lambda(\alpha)) \int_{\partial S^n_+}
\psi^2 .
\end{align}
But from \eqref{mult} and \eqref{halpha}
$$
\int_{\partial S^n_+} \vert\psi\vert^{p+1}
\geq
\lambda(\alpha) \int_{\partial S^n_+} \psi^2
$$
Combined with \eqref{ineq1}, we find that
$$
\lambda(\alpha) p \leq \Lambda_{n,s}
$$
unless $\psi\equiv0$.
\hfill\qed
\section{Blow-down analysis}
\label{sect:Blow-down analysis}

%

%

\begin{proof}[Proof of Theorem~\ref{thmstable}]

Assume that $p>p_{S}(n)$. Take a solution $u$ of \eqref{p} which is stable outside the ball of radius $R_{0}$ and let $\bar u$ be its extension solving \eqref{xp}.

\noindent {\bf Step 1.}\quad$
\lim_{\lambda\to+\infty} E(\bar u,0;\lambda)<+\infty.
$

Since $E$ is nondecreasing, it suffices to show that $E(\bar u,0;\lambda)$ is bounded.
Write $E=E_{1}+E_{2}$, where $E_{1}$ is given by \eqref{E1} and
$$
E_{2}(\bar u;\lambda) = \lambda^{2s\frac{p+1}{p-1}-n-1}\frac{s}{p+1}\int_{\partial B(0,\lambda)\cap\r}t^{1-2s}\bar u^2\;d\sigma
$$
By Lemma \ref{lj},  $E_{1}$ is bounded. Since $E$ is nondecreasing,
$$
E(\bar u;\lambda ) \le\frac1\lambda \int_{\lambda}^{2\lambda} E(u;t)\;dt \le C +\lambda^{2s\frac{p+1}{p-1}-n-1}\int_{B_{2\lambda}\cap\r}t^{1-2s}\bar u^2.
$$
Applying Lemma \ref{l2esti}, we deduce that $E$ is bounded.

\medskip

\noindent {\bf Step 2.}\quad 
There exists a sequence $\lambda_{i}\to+\infty$ such that $(\bar u^{\lambda_{i}})$ converges weakly in $H^1_{loc}(\r;t^{1-2s}dxdt)$ to a function $\bar u^\infty$.

This follows from the fact that $(\bar u^{\lambda_{i}})$ is bounded in $H^1_{loc}(\r;t^{1-2s}dxdt)$ by Lemma \ref{lj}.

\medskip

\noindent {\bf Step 3.\quad}$\bar u^\infty$ is homogeneous

To see this, apply the scale invariance of $E$, its finiteness and the monotonicity formula: given $R_{2}>R_{1}>0$,
\begin{eqnarray*}
0
&=&\lim\limits_{n\to+\infty}E(\bar u;\lambda_{i} R_{2})-E(\bar u;\lambda_{i} R_{1})\\
&=&\lim\limits_{n\to+\infty}E(\bar u^{\lambda_{i}};R_{2})-E(\bar u^{\lambda_{i}};R_{1})\\
&\geq&\liminf\limits_{n\to+\infty}\int_{(B_{R_{2}}\setminus
B_{R_{1}})\cap\r}t^{1-2s}r^{2-n+\frac{4s}{p-1}}\left(\frac{2s}{p-1}\frac{\bar u^{\lambda_{i}}}{r} +\frac{\partial
\bar u^{\lambda_{i}}}{\partial r}\right)^2\;dx\,dt\\
&\geq&\int_{(B_{R_{2}}\setminus
B_{R_{1}})\cap\r}t^{1-2s}r^{2-n+\frac{4s}{p-1}}\left(\frac{2s}{p-1}\frac{\bar u^\infty}{r} +\frac{\partial
\bar u^\infty}{\partial r}\right)^2\;dx\,dt
\end{eqnarray*}
Note that in the last inequality we only used the weak convergence
of $(\bar u^{\lambda_{i}})$ to $\bar u^\infty$ in $H^1_{loc}(\r;t^{1-2s}dxdt)$. So,
\[\frac{2s}{p-1}\frac{\bar u^\infty}r
+\frac{\partial \bar u^\infty}{\partial
r}=0\quad a.e.~~\text{in}~~\r.\]
And so, $u^\infty$ is homogeneous.

\medskip

\noindent {\bf Step 4.\quad}$\bar u^\infty\equiv0$

Simply apply Theorem \ref{h}.

\medskip

\noindent {\bf Step 5.\quad}$(\bar u^{\lambda_{i}})$ converges strongly to zero in $H^1(B_{R}\setminus B_{\eps};t^{1-2s}dxdt)$ and $(u^{\lambda_{i}})$ converges strongly to zero in $L^{p+1}(B_{R}\setminus B_{\eps})$ for all $R>\epsilon>0$. Indeed, by  Steps 2 and 3, $(\bar u^{\lambda_{i}})$ is bounded in $H^1_{loc}(\r;t^{1-2s}dxdt)$ and converges weakly to $0$. 
It follows that  $(\bar u^{\lambda_{i}})$ converges strongly to $0$ in $L^2_{loc}(\r;t^{1-2s}dxdt)$. Indeed, by the standard Rellich-Kondrachov theorem and a diagonal argument, passing to a subsequence we obtain
$$
\int_{\r \cap (B_R\setminus A)} t^{1-2s} |\bar u^{\lambda_{i}}|^2 \, dx dt\to0 ,
$$
as $i\to \infty$, for any  $B_R = B_R(0) \subset \R^{n+1}$ and $A$ of the form  $A = \{ (x,t)\in\r: 0<t< r/2\}$, where $R,r>0$. 
By \cite[Theorem 1.2]{fks},
$$
\int_{\r\cap B_r(x)}t^{1-2s} |\bar u^{\lambda_{i}}|^2 \, dx dt
\leq 
C r^2 \int_{\r\cap B_r(x)}t^{1-2s} |\nabla \bar u^{\lambda_{i}}|^2 \, dx dt
$$
for any $x\in \br$, $|x|\leq R$, with a uniform constant $C$. Covering $B_R \cap A$ with half balls $B_r(x) \cap \r$, $x\in \br$ with finite overlap, we see that 
$$
\int_{ B_R\cap A} t^{1-2s} |\bar u^{\lambda_{i}}|^2 \, dx dt
\leq  C r^2 \int_{ B_R\cap A} t^{1-2s} |\nabla \bar u^{\lambda_{i}}|^2 \, dx dt \leq C r^2,
$$
and from this we conclude  that  $(\bar u^{\lambda_{i}})$ converges strongly to $0$ in $L^2_{loc}(\r;t^{1-2s}dxdt)$. 

Now, using  \eqref{grad}, $(\bar u^{\lambda_{i}})$ converges strongly to $0$ in $H^1_{loc}(\r\setminus\{0\};t^{1-2s}dxdt)$ and by \eqref{last}, the convergence also holds in  $L^{p+1}_{loc}(\R^n\setminus\{0\})$.
%
%

\medskip

\noindent {\bf Step 6.\quad} $\bar u\equiv0$.

Indeed,
\begin{align*}
E_{1}(\bar u;\lambda) &= E_{1}(\bar u^\lambda;1) =
\int_{\r\cap B_{1}} t^{1-2s}\frac{\vert\nabla\bar u^\lambda\vert^2}{2}dx\,dt - \int_{\br\cap B_{1}}\frac{\kappa_{s}}{p+1}\vert \bar u^\lambda\vert^{p+1}dx\\
&=\int_{\r\cap B_{\epsilon}} t^{1-2s}\frac{\vert\nabla\bar u^\lambda\vert^2}{2}dx\,dt - \int_{\br\cap B_{\epsilon}}\frac{\kappa_{s}}{p+1}\vert \bar u^\lambda\vert^{p+1}dx+\\
&\int_{\r\cap B_{1}\setminus B_{\epsilon}} t^{1-2s}\frac{\vert\nabla\bar u^\lambda\vert^2}{2}dx\,dt - \int_{\br\cap B_{1}\setminus B_{\epsilon}}\frac{\kappa_{s}}{p+1}\vert \bar u^\lambda\vert^{p+1}dx\\
&=\eps^{n-2s\frac{p+1}{p-1}}E_{1}(\bar u;\lambda\eps) + \int_{\r\cap B_{1}\setminus B_{\epsilon}} t^{1-2s}\frac{\vert\nabla\bar u^\lambda\vert^2}{2}dx\,dt - \int_{\br\cap B_{1}\setminus B_{\epsilon}}\frac{\kappa_{s}}{p+1}\vert \bar u^\lambda\vert^{p+1}dx\\
&\le C\eps^{n-2s\frac{p+1}{p-1}} + \int_{\r\cap B_{1}\setminus B_{\epsilon}} t^{1-2s}\frac{\vert\nabla\bar u^\lambda\vert^2}{2}dx\,dt - \int_{\br\cap B_{1}\setminus B_{\epsilon}}\frac{\kappa_{s}}{p+1}\vert \bar u^\lambda\vert^{p+1}dx
\end{align*}
Letting $\lambda\to+\infty$ and then $\eps\to0$, we deduce that
$
\lim_{\lambda\to+\infty}E_{1}(\bar u;\lambda) =0.
$
Using the monotonicity of $E$,
\begin{multline*}
E(\bar u;\lambda) \le \frac1\lambda\int_{\lambda}^{2\lambda}E(t)\;dt\le \sup_{[\lambda,2\lambda]}E_{1} + C\lambda^{-n-1+2s\frac{p+1}{p-1}}\int_{B_{2\lambda}\setminus B_{\lambda}}\bar u^2\\
\end{multline*}
and so
$
\lim_{\lambda\to+\infty}E(\bar u;\lambda) =0.
$ Since $u$ is smooth, we also have $E(\bar u;0)=0$. Since $E$ is monotone, $E\equiv 0$ and so $\bar u$ must be homogeneous, a contradiction unless $\super u\equiv0$.
\end{proof}

\section{Construction of radial entire stable solutions}
\label{sect:Construction of radial entire stable solutions}

Let $\bar u_s$ denote the extension of the singular solution $u_s$ \eqref{sing sol} to $\r$ defined by 
$$
\bar u_s(X)= \int_{\R^n} P(X,y)u(y)\;dy .
$$
Let $B_{1}$ denote the unit ball in $\R^{n+1}$ and for $\lambda \geq 0$, consider
\begin{equation}
\label{prob ball}
\left \{
\begin{aligned}
{\rm div}\, (t^{1-2s} \nabla u)&=0&& 
\text{in } B_1\cap\R^{n+1}_{+}
\\
u&= \lambda \bar u_s&&
\text{on } \partial B_1\cap\R^{n+1}_{+}
\\
- \lim_{t\to0}(t^{1-2s}u_t)& = \kappa_s u^p&&
\text{ on $B_{1}\cap\{t=0\}$}.
\end{aligned}\right. 
\end{equation}
Take $\lambda\in(0,1)$. Since $u_s$ is a positive supersolution of \eqref{prob ball}, there exists a
minimal solution $u=u_\lambda$. By minimality, the family $(u_{\lambda})$ is nondecreasing and $u_\lambda$ is axially symmetric, that is, $u_\lambda(x,t)= u_\lambda(r,t)$ with $r=|x| \in [0,1]$. In addition, for a fixed value $\lambda\in(0,1)$, $u_\lambda$ is bounded, as can be proved by the truncation method of \cite{brezis-cazenave-martel-ramiandrisoa}, see also \cite{davila-handbook} and radially decreasing by the moving plane method (see  \cite{capella-davila-dupaigne-sire} for a similar setting).
From now on let us assume that $p_{S}(n)<p$ and
\begin{equation*}
p \frac{\Gamma(\frac n2-\frac{s}{p-1}) \Gamma(s+\frac{s}{p-1})}{\Gamma(\frac{s}{p-1}) \Gamma(\frac{n-2s}{2} - \frac{s}{p-1})} \leq \frac{\Gamma(\frac{n+2s}{4})^2}{\Gamma(\frac{n-2s}{4})^2},
\end{equation*}
which means that the singular solution $u_s$ is stable. Then, $u_{\lambda}\uparrow u_{s}$ as  $\lambda\uparrow 1$, using the classical convexity argument in \cite{brezis-vazquez} (see also Section 3.2.2 in \cite{dupaigne}).
Let $\lambda_j\uparrow 1$ and 
$$
m_j = \|u_{\lambda_j}\|_{L^\infty} =u_{\lambda_j}(0),
\quad R_j = m_j^{\frac{p-1}{2s}} ,
$$
so that $m_j$, $R_j\to\infty$ as $j\to\infty$.
Set 
$$
v_j(x) = m_j^{-1} u_{\lambda_j}(x/R_j) .
$$
Then $0\leq v_j\leq 1$ is a bounded solution of
\begin{equation*}
\left \{
\begin{aligned}
{\rm div}\, (t^{1-2s} \nabla v_j)&=0&& 
\text{in } B_{R_j}\cap\R^{n+1}_{+}
\\
v_j&= \lambda_j \bar u_s&&
\text{on } \partial B_{R_j} \cap\R^{n+1}_{+}\\
- \lim_{t\to0}(t^{1-2s}(v_j)_t)& = \kappa_s v_j^p&&
\text{ on $B_{R_j}\cap\{t=0\}$}.
\end{aligned}\right. 
\end{equation*}
Moreover $v_j\leq \bar u_s$ in $B_{R_{j}}\cap\r$ and $v_j(0)=1$.
Using elliptic estimates we find (for a subsequence) that $(v_j)$ converges uniformly on compact sets of $\overline \R_+^{n+1}$ to a function $v$ that is axially symmetric and solves
\begin{equation*}
\left \{
\begin{aligned}
{\rm div}\, (t^{1-2s} \nabla v)&=0&& 
\text{in } \r
\\
- \lim_{t\to0}(t^{1-2s}v_t)& = \kappa_s v^p&&
\text{ on }\R^n\times\{0\}.
\end{aligned}\right. 
\end{equation*}
Moreover $0\leq v \leq 1$, $v(0)=1$ and $v\leq \bar u_s$. This $v$ restricted to $\R^n\times\{0\}$ is a radial, bounded, smooth solution of \eqref{p} and from  $v\leq \bar u_s$ we deduce that $v$ is stable.

\bigskip

\noindent
{\bf Acknowledgments:} 
J. D\'avila is
supported by Fondecyt 1130360 and Fondo Basal CMM, Chile.
The research of J. Wei is partially supported by NSERC of Canada. L. Dupaigne is partially supported by ERC grant Epsilon.

\bibliographystyle{amsalpha}

\begin{bibdiv}
\begin{biblist}

\end{biblist}
\end{bibdiv}


\begin{bibdiv}
\begin{biblist}


\bib{brezis-cazenave-martel-ramiandrisoa}{article}{
   author={Brezis, Ha{\"{\i}}m},
   author={Cazenave, Thierry},
   author={Martel, Yvan},
   author={Ramiandrisoa, Arthur},
   title={Blow up for $u_t-\Delta u=g(u)$ revisited},
   journal={Adv. Differential Equations},
   volume={1},
   date={1996},
   number={1},
   pages={73--90},
   issn={1079-9389},
   review={\MR{1357955 (96i:35063)}},
}

\bib{brezis-vazquez}{article}{
   author={Brezis, Haim},
   author={V{\'a}zquez, Juan Luis},
   title={Blow-up solutions of some nonlinear elliptic problems},
   journal={Rev. Mat. Univ. Complut. Madrid},
   volume={10},
   date={1997},
   number={2},
   pages={443--469},
   issn={0214-3577},
   review={\MR{1605678 (99a:35081)}},
}

\bib{cabre-sire}{article}{
   author={Cabr{\'e}, Xavier},
   author={Sire, Yannick},
   title={Nonlinear equations for fractional Laplacians, I: Regularity,
   maximum principles, and Hamiltonian estimates},
   journal={Ann. Inst. H. Poincar\'e Anal. Non Lin\'eaire},
   volume={31},
   date={2014},
   number={1},
   pages={23--53},
   issn={0294-1449},
   review={\MR{3165278}},
   doi={10.1016/j.anihpc.2013.02.001},
}

\bib{cgs}{article}{
   author={Caffarelli, Luis A.},
   author={Gidas, Basilis},
   author={Spruck, Joel},
   title={Asymptotic symmetry and local behavior of semilinear elliptic
   equations with critical Sobolev growth},
   journal={Comm. Pure Appl. Math.},
   volume={42},
   date={1989},
   number={3},
   pages={271--297},
   issn={0010-3640},
   review={\MR{982351 (90c:35075)}},
   doi={10.1002/cpa.3160420304},
}

\bib{caffarelli-silvestre-extension}{article}{
   author={Caffarelli, Luis},
   author={Silvestre, Luis},
   title={An extension problem related to the fractional Laplacian},
   journal={Comm. Partial Differential Equations},
   volume={32},
   date={2007},
   number={7-9},
   pages={1245--1260},
   issn={0360-5302},
   review={\MR{2354493 (2009k:35096)}},
   doi={10.1080/03605300600987306},
}

\bib{carfarelli-roquejoffre-savin}{article}{
   author={Caffarelli, L.},
   author={Roquejoffre, J.-M.},
   author={Savin, O.},
   title={Nonlocal minimal surfaces},
   journal={Comm. Pure Appl. Math.},
   volume={63},
   date={2010},
   number={9},
   pages={1111--1144},
   issn={0010-3640},
   review={\MR{2675483 (2011h:49057)}},
   doi={10.1002/cpa.20331},
}

\bib{capella-davila-dupaigne-sire}{article}{
   author={Capella, Antonio},
   author={D{\'a}vila, Juan},
   author={Dupaigne, Louis},
   author={Sire, Yannick},
   title={Regularity of radial extremal solutions for some non-local
   semilinear equations},
   journal={Comm. Partial Differential Equations},
   volume={36},
   date={2011},
   number={8},
   pages={1353--1384},
   issn={0360-5302},
   review={\MR{2825595 (2012h:35361)}},
   doi={10.1080/03605302.2011.562954},
}

\bib{cho}{article}{
   author={Chen, Wenxiong},
   author={Li, Congming},
   author={Ou, Biao},
   title={Classification of solutions for an integral equation},
   journal={Comm. Pure Appl. Math.},
   volume={59},
   date={2006},
   number={3},
   pages={330--343},
   issn={0010-3640},
   review={\MR{2200258 (2006m:45007a)}},
   doi={10.1002/cpa.20116},
}

\bib{chipot-chlebik-shafrir}{article}{
   author={Chipot, M.},
   author={Chleb{\'{\i}}k, M.},
   author={Fila, M.},
   author={Shafrir, I.},
   title={Existence of positive solutions of a semilinear elliptic equation
   in $\bold R^n_{+}$ with a nonlinear boundary condition},
   journal={J. Math. Anal. Appl.},
   volume={223},
   date={1998},
   number={2},
   pages={429--471},
   issn={0022-247X},
   review={\MR{1629293 (99h:35060)}},
   doi={10.1006/jmaa.1998.5958},
}

\bib{davila-handbook}{article}{
   author={D{\'a}vila, J.},
   title={Singular solutions of semi-linear elliptic problems},
   conference={
      title={Handbook of differential equations: stationary partial
      differential equations. Vol. VI},
   },
   book={
      series={Handb. Differ. Equ.},
      publisher={Elsevier/North-Holland, Amsterdam},
   },
   date={2008},
   pages={83--176},
   review={\MR{2569324 (2010k:35179)}},
   doi={10.1016/S1874-5733(08)80019-8},
}

\bib{ddf}{article}{
   author={D{\'a}vila, Juan},
   author={Dupaigne, Louis},
   author={Farina, Alberto},
   title={Partial regularity of finite Morse index solutions to the
   Lane-Emden equation},
   journal={J. Funct. Anal.},
   volume={261},
   date={2011},
   number={1},
   pages={218--232},
   issn={0022-1236},
   review={\MR{2785899 (2012e:35090)}},
   doi={10.1016/j.jfa.2010.12.028},
}


\bib{davila-dupaigne-montenegro}{article}{
   author={D{\'a}vila, Juan},
   author={Dupaigne, Louis},
   author={Montenegro, Marcelo},
   title={The extremal solution of a boundary reaction problem},
   journal={Commun. Pure Appl. Anal.},
   volume={7},
   date={2008},
   number={4},
   pages={795--817},
   issn={1534-0392},
   doi={10.3934/cpaa.2008.7.795},
}

\bib{ddww}{article}{
   author={D{\'a}vila, Juan},
   author={Dupaigne, Louis},
   author={Kelei Wang},
   author={Juncheng Wei},
   title={  A Monotonicity Formula and a Liouville-type Theorem for a Fourth Order Supercritical Problem },
   journal={Advances in Mathematics, to appear}
}

\bib{dupaigne}{book}{
   author={Dupaigne, Louis},
   title={Stable solutions of elliptic partial differential equations},
   series={Chapman \& Hall/CRC Monographs and Surveys in Pure and Applied
   Mathematics},
   volume={143},
   publisher={Chapman \& Hall/CRC, Boca Raton, FL},
   date={2011},
   pages={xiv+321},
   isbn={978-1-4200-6654-8},
   review={\MR{2779463 (2012i:35002)}},
   doi={10.1201/b10802},
}

\bib{fks}{article}{
   author={Fabes, Eugene B.},
   author={Kenig, Carlos E.},
   author={Serapioni, Raul P.},
   title={The local regularity of solutions of degenerate elliptic
   equations},
   journal={Comm. Partial Differential Equations},
   volume={7},
   date={1982},
   number={1},
   pages={77--116},
   issn={0360-5302},
   review={\MR{643158 (84i:35070)}},
   doi={10.1080/03605308208820218},
}

\bib{fall}{article}{
   author={Fall, Mouhamed Moustapha},
   title={ Semilinear elliptic equations for the fractional Laplacian with Hardy potential},
   review={ http://arxiv.org/abs/1109.5530},
}


\bib{fall-felli}{article}{
   author={Fall, Mouhamed Moustapha},
     author={Felli, Veronica},
   title={Unique continuation property and local asymptotics of solutions to fractional elliptic equations},
   review={\quad http://arxiv.org/abs/1301.5119},
}

\bib{f}{article}{
   author={Farina, Alberto},
   title={On the classification of solutions of the Lane-Emden equation on
   unbounded domains of $\Bbb R^N$},
   language={English, with English and French summaries},
   journal={J. Math. Pures Appl. (9)},
   volume={87},
   date={2007},
   number={5},
   pages={537--561},
   issn={0021-7824},
   review={\MR{2322150 (2008c:35070)}},
   doi={10.1016/j.matpur.2007.03.001},
}

\bib{frank-lenzmann-silvestre}{article}{
author = {Rupert L. Frank},
author = {Enno Lenzmann},
author = {Luis Silvestre},
title = {Uniqueness of radial solutions for the fractional laplacian},
review={\quad http://arxiv.org/abs/1302.2652},
}



\bib{gs}{article}{
   author={Gidas, B.},
   author={Spruck, J.},
   title={A priori bounds for positive solutions of nonlinear elliptic
   equations},
   journal={Comm. Partial Differential Equations},
   volume={6},
   date={1981},
   number={8},
   pages={883--901},
   issn={0360-5302},
   review={\MR{619749 (82h:35033)}},
   doi={10.1080/03605308108820196},
}

\bib{h}{article}{
   author={Harada, Junichi},
   title={Positive solutions to the Laplace equation with
nonlinear boundary conditions on the half space},
   review={preprint},
}

\bib{herbst}{article}{
   author={Herbst, Ira W.},
   title={Spectral theory of the operator
   $(p^{2}+m^{2})^{1/2}-Ze^{2}/r$},
   journal={Comm. Math. Phys.},
   volume={53},
   date={1977},
   number={3},
   pages={285--294},
   issn={0010-3616},
   review={\MR{0436854 (55 \#9790)}},
}



\bib{jl}{article}{
   author={Joseph, D. D.},
   author={Lundgren, T. S.},
   title={Quasilinear Dirichlet problems driven by positive sources},
   journal={Arch. Rational Mech. Anal.},
   volume={49},
   date={1972/73},
   pages={241--269},
   issn={0003-9527},
}

\bib{yy}{article}{
   author={Li, Yan Yan},
   title={Remark on some conformally invariant integral equations: the
   method of moving spheres},
   journal={J. Eur. Math. Soc. (JEMS)},
   volume={6},
   date={2004},
   number={2},
   pages={153--180},
   issn={1435-9855},
   review={\MR{2055032 (2005e:45007)}},
}

\bib{mo}{article}{
   author={Mol{\v{c}}anov, S. A.},
   author={Ostrovski{\u\i}, E.},
   title={Symmetric stable processes as traces of degenerate diffusion
   processes. },
   language={Russian, with English summary},
   journal={Teor. Verojatnost. i Primenen.},
   volume={14},
   date={1969},
   pages={127--130},
   issn={0040-361x},
   review={\MR{0247668 (40 \#931)}},
}



\bib{P-Q-S}{article}{
   author={Pol{\'a}{\v{c}}ik, Peter},
   author={Quittner, Pavol},
   author={Souplet, Philippe},
   title={Singularity and decay estimates in superlinear problems via
   Liouville-type theorems. I. Elliptic equations and systems},
   journal={Duke Math. J.},
   volume={139},
   date={2007},
   number={3},
   pages={555--579},
   issn={0012-7094},
   review={\MR{2350853 (2009b:35131)}},
   doi={10.1215/S0012-7094-07-13935-8},
}



\bib{rs2}{article}{
   author={Ros-Oton, Xavier},
   author={Serra, Joaquim},
   title={The Pohozaev identity for the fractional Laplacian},
   review={ http://arxiv.org/abs/1207.5986},
}

\bib{spitzer}{article}{
   author={Spitzer, Frank},
   title={Some theorems concerning $2$-dimensional Brownian motion},
   journal={Trans. Amer. Math. Soc.},
   volume={87},
   date={1958},
   pages={187--197},
   issn={0002-9947},
   review={\MR{0104296 (21 \#3051)}},
}

\bib{kw}{article}{
   author={Wang, Kelei},
   title={Partial regularity of stable solutions to the supercritical
   equations and its applications},
   journal={Nonlinear Anal.},
   volume={75},
   date={2012},
   number={13},
   pages={5238--5260},
   issn={0362-546X},
   review={\MR{2927586}},
   doi={10.1016/j.na.2012.04.041},
}
\end{biblist}
\end{bibdiv}

\end{document}